\documentclass[letterpaper,10pt]{amsart}
\usepackage{indentfirst} 
\usepackage{amssymb}
\usepackage{amsmath}
\usepackage{mathabx} 
\usepackage{amsthm}
\usepackage{thmtools}
\usepackage{bbm}             
\usepackage[colorlinks=true]{hyperref}  
\usepackage[usenames,dvipsnames]{xcolor} 
\usepackage{comment}

\newtheorem{theorem}{Theorem}[section]
\newtheorem{definition}[theorem]{Definition}
\newtheorem{proposition}[theorem]{Proposition}
\newtheorem{corollary}[theorem]{Corollary}
\newtheorem{lemma}[theorem]{Lemma}

\newtheorem{remark}[theorem]{Remark}

\newtheorem*{theorem-non}{Theorem}

\declaretheorem[name=Acknowledgements,numbered=no]{ack}

\theoremstyle{definition}
\newtheorem{example}[theorem]{Example}

\newcommand{\M}{\mathcal{M}}
\newcommand{\R}{\mathbb{R}}

\newcommand{\N}{\mathbb{N}}

\newcommand{\cF}{\mathcal{F}}
\newcommand{\cG}{\mathcal{G}}

\newcommand{\cO}{\mathcal{O}}
\newcommand{\cQ}{\mathcal{Q}}

\def\phi{\varphi}
\def\R{{\mathbb R}}

\def\N{{\mathbb N}}

\def\O{{\mathcal O}}

\def\F{{\mathcal F}}

\def\M{{\mathcal M}}

\def\S{{\mathcal S}}

\def\g{{\mathfrak{g}}}

\def\le{\leqslant}
\def\ge{\geqslant}

\def\e{\epsilon}

\def\F{\mathcal{F}}
\def\M{\mathcal{M}}

\begin{document}

\title[Ergodic optimization and zero temperature limits]{Ergodic optimization and zero temperature limits in negative curvature}
\date{\today}

\author[F. Riquelme]{Felipe Riquelme}
\address{IMA, Pontificia Universidad Cat\'olica de Valpara\'iso, Blanco Viel 596, Valpara\'iso, Chile.}
\email{\href{felipe.riquelme@pucv.cl}{felipe.riquelme@pucv.cl}}
\urladdr{\href{http://ima.ucv.cl/academicos/felipe-riquelme/}{http://ima.ucv.cl/academicos/felipe-riquelme/}}
\thanks{F.R. was partially supported by PUCV DI Emergente N$^o$039.401/19 and FONDECYT Iniciaci\'on N$^o$11190461}

 \author[A.~Velozo]{Anibal Velozo}  \address{Department of Mathematics, Yale University, New Haven, CT 06511, USA.}
\email{\href{anibal.velozo@gmail.com}{anibal.velozo@yale.edu}}
\urladdr{\href{https://gauss.math.yale.edu/~av578/}{https://gauss.math.yale.edu/~av578/}}

\begin{abstract}
In this paper we study aspects of the ergodic theory of the geodesic flow on a non-compact negatively curved manifold. It is a well known fact that every continuous potential on a compact metric space has a maximizing measure. Unfortunately, for non-compact spaces this fact is not longer true. For the geodesic flow we provide a criterion that ensures the existence of a maximizing measure for uniformly continuous potentials.  We prove that the only obstruction to the existence of a maximizing measure is the full escape of mass phenomenon. To the best of our knowledge, this is the first general result on the existence of maximizing measures for non-compact topological spaces which does not require the potential to be coercive. We study zero temperature limits of equilibrium measures for a suitable family of potentials. We prove some convergence and divergence results for the limiting behaviour of such measures. Among some consequences we obtain that the geodesic flow has the intermediate entropy property and that equilibrium states are dense in the space of invariant probability measures. 
\end{abstract}

\maketitle

\section{Introduction}

Let $(X,T)$ be a topological dynamical system and $\phi:X\to\R$ a real-valued function, or potential. An important goal in the subject of \emph{ergodic optimization} is to describe the largest possible space average of $\phi$ among invariant probability measures. More specifically, we aim to understand the function 
$$\beta(\phi):=\sup_{\mu\in \M(T)}\int \phi d\mu,$$
where $\M(T)$ is the space of $T-$invariant probability measures, and to describe the measures that achieve the supremum (if those measures exist). A measure $m\in \M(T)$ such that $\beta(\phi)=\int \phi dm$ is called a \emph{maximizing measure} for the potential $\phi$. Natural questions in this context are, for instance, what conditions ensure the existence of a maximizing measure, and if such measures exist, what can be said about their dynamical properties and support. For a general overview on the history and discussion of several aspects of ergodic optimization we refer the reader to \cite{j1,bll,  j, bo}. 

When $X$ is a compact topological space and the maps $T$ and $\phi$ are continuous, the weak$^*$ compactness of $\M(T)$ ensures the existence of maximizing measures. In this context it is known that generic continuous functions have a unique ma\-ximizing measure (see \cite[Theorem 3.2]{j1}), and that for expanding transformations, the maximizing measure
of a generic Lipschitz function is supported on a single periodic orbit (see \cite[Theorem A]{c} and \cite[Theorem A]{hlmxz1}). For Axiom A flows it was recently proved that generic H\"older potentials have a unique maximizing measure supported on a periodic orbits (see \cite[Theorem A]{hlmxz}). 

Unlike the compact case, if $X$ is a non-compact topological space, it is possible that continuous (or Lipschitz) potentials do not have any maximizing measure (see \cite[Theorem 1.1]{iommi}, or Example \ref{ex:nomea}).   Despite this substantial difference between the compact and non-compact case, there are situations where the existence of maximizing measures has been obtained for non-compact spaces (mostly in the setting of symbolic dynamics). For transitive (one-sided) countable Markov shifts $(\Sigma,\sigma)$ and coercive potentials $\phi:\Sigma\to\R$ with bounded variation (equivalently, with summable variations and finite first variation), the existence of maximizing measures was obtained in \cite{bf} (see also \cite{jmu}). Roughly speaking, a potential $\phi$ is coercive if $\phi(x)$ goes to $-\infty$ as $x$ diverges. In particular, a coercive potential does not oscillate at infinity. For dynamical systems such as the full shift on a countable alphabet (which has infinite topological entropy) such assumption is not too restrictive: a  potential with bounded variation and finite pressure is necessarily coercive. 

In this work we study a class of non-compact dynamical systems which are geometric in nature: the geodesic flow over a non-compact negatively curved manifold. Our first goal is to provide a criterion that ensures the existence of maximizing measures for uniformly continuous potentials. In this context the dynamical system has finite topological entropy and the coercive assumption is restrictive; we ultimately seek for a criterion that applies to potentials that oscillate at infinity. 

In order to state our main results let us first introduce some notation (for details we refer the reader to Section \ref{prelim}). Let $(M,g)$ be a pinched negatively curved complete Riemannian manifold, $X:=T^1M$ the unit tangent bundle of $(M,g)$ and $(\g_t)_{t\in\R}$ the geodesic flow on $X$. Let $\M(\g)$ (resp. $\M_{\le1}(\g)$) be the space of $\g$-invariant probability (resp. sub-probability) measures on $X$. We say that a sequence $(\mu_n)_n$ in $\M(\g)$ \emph{converges vaguely} to $\mu\in \M_{\le1}(\g)$ if $\lim_{n\to+\infty}\int f d\mu_n=\int f d\mu$, for any compactly supported continuous function $f:X\to\R$. The mass of a measure $\mu\in \M_{\le1}(\g)$ is $\|\mu\|:=\mu(X)$.

For the rest of the introduction we assume that the non-wandering set of the geodesic flow on $X$ is non-compact, that is, $(M,g)$ is not convex cocompact. Our next definition plays an important role in this work. 

\begin{definition}\label{def:Minf}  For a continuous potential $\phi$ we define 
$$\beta_\infty(\phi):=\sup_{(\mu_n)_n\to 0}\limsup_{n\to+\infty}\int \phi d\mu_n,$$
where the supremum runs over sequences $(\mu_n)_n$ in $\M(\g)$ which converge vaguely to the zero measure. 
\end{definition}

Since $(M,g)$ is not convex cocompact, there are sequences in $\M(\g)$ that converge vaguely to the zero measure (see \cite[Theorem 4.19]{v2}), hence $\beta_\infty(\phi)$ is well defined. The importance of the quantity $\beta_\infty(\phi)$ lies in the following result and its corollary.

\begin{theorem}\label{main:a} Let $(M,g)$ be a complete pinched negatively curved manifold and $X=T^1M$ the unit tangent bundle.  Let $(\mu_n)_n$ be a sequence of ergodic measures in $\M(\g)$ which converges vaguely to $\mu\in \M_{\le1}(\g)$. Suppose that $\phi:X\to\R$ is a bounded from above and uniformly continuous potential. Then 
$$\limsup_{n\to+\infty}\int \phi d\mu_n\le \int \phi d\mu+(1-\|\mu\|)\beta_\infty(\phi).$$ 
\end{theorem}

\noindent
\textbf{Corollary \ref{cor:gapM}} \emph{Let $\phi$ be a bounded from above and uniformly continuous potential such that $\beta_\infty(\phi)<\beta(\phi)$. Then there exists a maximizing measure for $\phi$}.\\

To the best of our knowledge, this is the first general result about the existence of maximizing measures in the non-compact setting that goes beyond the coercive case (and symbolic dynamics). Moreover, our approach works for uniformly continuous potentials, which is a regularity assumption weaker than summable variations, a standard hypothesis in the symbolic setting. Note that if $\phi$ is coercive, then $\beta_\infty(\phi)=-\infty$.  The condition $\beta_\infty(\phi)<\beta(\phi)$ can be interpreted as a type of dynamical gap between space averages on a large compact part and its complement. Moreover, it impose no constraint on the behavior of $\phi$ at infinity.

Another consequence of Theorem \ref{main:a} is the following. Let $(\mu_n)_n$ be a sequence in $\M(\g)$ such that $\lim_{n\to+\infty}\int \phi d\mu_n=\beta(\phi)$, and suppose $\phi$ does not have a maximizing measure. Then $(\mu_n)_n$ converges vaguely to the zero measure. Equivalently, the only obstruction to the existence of a maximizing measure is the full escape of mass of sequences with space average approximating $\beta(\phi)$. This is analogous to results known for the entropy of sequences of invariant measure (see \cite[Theorem 5.2]{v2}). For further details we refer the reader to Proposition \ref{prop:a}.

Among measures with maximal space average some are  physically more relevant than others. This distinction among maximizing measures is done via the pressure function and involves thermodynamic formalism. Let us be more precise. Let $\phi$  be a potential such that $t\phi$ admits a unique equilibrium state $m_{t\phi}$ for large enough $t\in\R$. A \emph{ground state at zero temperature} for $\phi$ is an accumulation point of the sequence $(m_{t\phi})_t$ as $t$ goes to infinity. In the compact case it is well known that ground states at zero temperature are maximizing measures with largest  entropy among those.

For subshifts of finite type, it is proved in \cite{bre} that locally constant potentials have a unique ground state at zero temperature, or equivalently, that the sequence $(m_{t\phi})_t$ converges as $t$ goes to $+\infty$, whenever $\phi$ is a finite-range potential.  For countable Markov shifts with the BIP property and Markov potentials, the same result was obtained in \cite{kem} (see also \cite{mo}).  
The situation is different for more general potentials. Still in the context of symbolic dynamics, H\"older potentials for which the sequence of equilibrium states $(m_{t\phi})_t$   diverges have been constructed in \cite{ch} and \cite{crl}. For countable Markov shifts, the behavior of the sequence $(m_{t\phi})_t$ has also been studied without the  locally constant assumption. For summable potentials with bounded variation and finite Gurevich pressure, it is proved in \cite{fv} that ground states at zero temperature are maximizing measures and that   the entropy map $t\mapsto h(m_{t\phi})$ is continuous at $+\infty$. 

In this work we study zero temperature limits of a particular class of potentials. We denote by $\F$ the family of H\"older-continuous  potentials $\phi$ vanishing at infinity (see Section \ref{prelim}) and such that $\beta(\phi)>0$. For $\phi\in\F$ and large enough $t\in\R$, the potential $t\phi$ admits a unique equilibrium measure $m_{t\phi}$. In Section \ref{sec:gz} we prove the following result. 

\begin{theorem}\label{main:b} Let $\phi\in\F$. Every ground state at zero temperature of $\phi$ is a maximizing measure. Moreover, they all have the same measure-theoretic entropy.
\end{theorem}

Among the consequences of Theorem \ref{main:b} we  obtain that the geodesic flow verifies the intermediate entropy property (see Corollary \ref{cor:iep}), and that the set of equilibrium states of potentials in $\F$ is dense in $\M_{\le1}(\g)$, relative to the vague topology (see Corollary \ref{density}). 
Finally, we construct potentials in $\F$ for which the zero temperature limits converge and diverge. Since zero temperature limits are probability measures (see Theorem \ref{main:b}), the divergence of $(m_{t\phi})_t$ is equivalent to say that the sequence has at least two  probability measures as accumulation points. 

\begin{theorem}\label{main:c} There are potentials $\phi\in \F$ such that $(m_{t\phi})_t$ converges as $t$ goes to $+\infty$, and potentials such that $(m_{t\phi})_t$ diverges as $t$ goes to $+\infty$. 
\end{theorem}

\noindent {\bf Structure of the paper:} In Section \ref{prelim} we collect some known facts about the geodesic flow on a negatively curved manifold and set up notation used in following sections. In Section  \ref{sec3} we prove Theorem \ref{main:a} and some consequences of it. Finally, in Section \ref{sec:gz} we study zero temperature limits of potentials in $\F$, where we prove Theorem \ref{main:b} and Theorem \ref{main:c}. 

\begin{ack} The authors would like to thank G. Iommi for careful reading and a wealth of useful comments on the subject of this article. We would also like to thank Jairo Bochi for useful comments and clarifications. 
 \end{ack}

\section{Preliminaries}\label{prelim}

\subsection{Geometry}

Let $(M,g)$ be a complete Riemannian manifold. The unit tangent bundle of $(M,g)$ is defined as $T^1M=\{v\in TM: \|v\|_g=1\}$. For simplicity, we set $X:=T^1M$. Since $(M,g)$ is complete, the geodesic flow is well defined for all time and we denote it by $\g:=(g_t)_{t\in\R}$, where $g_t:X\to X$ is the time $t$ evolution of the geodesic flow. The geodesic flow on $(M,g)$ is the 
continuous time dynamical system $(X,\g)$. From now on we will always assume that $(M,g)$ has pinched negative sectional curvatures $-b^2\leq K_g\leq -a^2$, with $0<a<b<+\infty$, and that the partial derivatives of $K_g$ are uniformly bounded. For general facts about the geodesic flow we refer the reader to  \cite{eb, pps}.

The \emph{non-wandering set} of the geodesic flow is the set of points $x\in X$ so that for every open set $U$ containing $x$, and $N>0$, there exists $n>N$ so that $g_{-n}(U)\cap U\ne \emptyset$. The manifold $(M,g)$ is said to be \emph{convex cocompact} if the non-wandering set of the geodesic flow is compact. Since the dynamically relevant  part of the geodesic flow is supported on the non-wandering set of $X$, we will be mostly interested in the case where $(M,g)$ is not convex cocompact; if $(M,g)$ is convex cocompact then the dynamical system $(X,\g)$ is compact from the point of view of ergodic theory. 

 The universal cover of $M$ is denoted by $\widetilde{M}$ and let $\widetilde{X}=T^1\widetilde{M}$. Let $\pi:X\to M$ and $\widetilde{\pi}:\widetilde{X}\to\widetilde{M}$ be the canonical projections and $p:\widetilde{M}\to M$  the covering map. The metric on $M$ induced by the Riemannian metric $g$ is denoted by $d$. We define a metric on $X$, which we still denote by $d$, in the following way 
\begin{align}\label{distance} 
d(x,y)=\max_{t\in [0,1]} d(\pi g_t(x),\pi  g_t(y)),
\end{align}
for every $x,y\in X$. In a similar way we define a metric on $\widetilde{M}$ and $\widetilde{X}$, which we denote by $\widetilde{d}.$

We denote by $\partial_\infty \widetilde{M}$ the visual boundary of $\widetilde{M}$. The set $\widetilde{M}\cup\partial_\infty\widetilde{M}$ is the Gromov compactification of $\widetilde{M}$ and is homeomorphic to a closed ball. Using the ball model of real hyperbolic space we can identify this compactification with the round unit ball; this is the representation we have in mind even if the curvature is not constant.

The isometry group of $\widetilde{M}$ is denoted by Iso$(\widetilde{M})$.  It is well known that every isometry of $\widetilde{M}$ extends to a homeomorphism of $\widetilde{M}\cup\partial_\infty \widetilde{M}$ and that the fundamental group of $M$ acts--via Deck transformations--isometrically, freely and discontinuously on $\widetilde{M}$. We identify $\pi_1(M)$ with a subgroup $\Gamma\leqslant\text{Iso}(\widetilde{M})$. Since $\Gamma$ has no fixed points in $\widetilde{M}$, an isometry $\gamma\in\Gamma$ can be either hyperbolic or parabolic. A \emph{hyperbolic isometry} fixes  two points at infinity (pointwise) and the geodesic connecting those points (setwise). Such geodesic is called the \emph{axis} of the hyperbolic isometry. The axis of a hyperbolic isometry in $\Gamma$ projects to a closed geodesic on $M$. A \emph{parabolic isometry} fixes exactly one point at infinity. 

Let $G\leqslant\text{Iso}(\widetilde{M})$ be a discrete  torsion-free subgroup of isometries. Fix a reference point $o\in\widetilde{M}$. The \emph{limit set} of  $G$ is defined as $L(G)=\overline{G\cdot o}\setminus G\cdot o$, that is, the set of accumulation points of the $G-$orbit of $o$ in $\partial_\infty \widetilde{M}$. The limit set of $G$ is independent of the chosen reference point. The group $G$ is said to be \emph{non-elementary} if $\# L(G)>2$. If $G$ is non-elementary, then the set of points which are  fixed (pointwise) by a hyperbolic element in $G$ is dense in $L(G)$. For further details we refer the reader to  \cite[Section 3]{bow}.

An important number in this context is the \emph{critical exponent} $\delta(G)$ of $G$. It is defined by
$$\delta(G) = \limsup_{R\to+\infty}\frac{1}{R}\log\#\{g \in G : d(o,g \cdot o)\leq R\}.$$
Note that by triangle inequality the critical exponent does not depends on the re\-ference point $o$. The critical exponent can also be described as the critical exponent of convergence of the Poincar\'e series
$$P_G(s)=\sum_{g\in G} \exp(-sd(o,g\cdot o)).$$
Indeed, the series $P_G(s)$ converges for $s>\delta(G)$ and diverges for $s<\delta(G)$. The group $G$ is called \emph{divergent} (resp. \emph{convergent}) if $P_G(s)$ diverges (resp. converges) at $s=\delta(G)$. The following proposition provides useful information about critical exponents (see \cite[Proposition 2]{dop})  which we will need in later sections. 

\begin{proposition}\label{prop:CE} Let $G$ be a discrete torsion-free subgroup of isometries. If $H$ is a divergent subgroup of $G$ and $L(H)\neq L(G)$, then $\delta(H)<\delta(G)$.
\end{proposition}

\begin{remark}\label{rk:CE} If $G$ is non-elementary we choose a hyperbolic element $h\in G$ and define $H=\langle h\rangle$. It is not difficult to prove that $\delta(H)=0$ and $L(H)\neq L(G)$, so by Proposition \ref{prop:CE} we conclude that $\delta(G)>0$. 
\end{remark}

\subsection{Ergodic theory}
We denote by $\M$ the set of all Borel non-negative measures on $X$. The mass of a measure $\mu\in\M$ is the number $\|\mu\|:=\mu(X)$. We say that $\mu$ is a sub-probability (resp. probability) measure if $\|\mu\|\in[0,1]$ (resp. $\|\mu\|=1$). A measure $\mu\in\M$ is $\g$-invariant if $\mu(g_t A)=\mu(A)$, for every $t\in\R$ and for every Borel set $A\subseteq X$. We denote by $\M(\g)$ the space of $\g$-invariant probability measures and $\M_{\le 1}(\g)$ the space of $\g$-invariant sub-probability measures in $\M$.  If $\mu\in \M_{\le 1}(\g)$ has positive mass, we define $\overline{\mu}:=\mu/\|\mu\|$, which is a probability measure. 

We say that $x\in X$ is a periodic point if $g_T(x)=x$, for some $T>0$.  Denote by $\cO(x)$ the periodic orbit associated to $x$. We define the periodic measure on $\cO(x)$, which we denote by $\mu_{\mathcal{O}(x)}$, by
\begin{align}\label{periodic}\mu_{\mathcal{O}(x)} = \frac{1}{T}\int_0^T \delta_{g_t(x)} dt,\end{align}
where $\delta_{y}$ denotes the Dirac measure on $y\in X$. Note that $\mu_{\mathcal{O}(x)}\in\M(\g)$ for any periodic point $x\in X$.

Let $C(X)$ be the space of continuous functions on $X$ and $C_{b}(X)$  the subset of bounded functions. The space of compactly supported continuous functions on $X$ is denoted by $C_c(X)$. We endow $C_b(X)$ and $C_c(X)$ with the uniform norm, that is, for $\varphi:X\to\R$ we set $\|\varphi\|_\infty=\sup_{x\in X} |\varphi(x)|$. 

We endow $\M(\g)$ with the \emph{narrow topology}. A sequence $(\mu_n)_n$ converges to $\mu$ in the narrow topology if for every $\varphi\in C_b(X)$, we have $\lim_{n\to+\infty}\int \varphi d\mu_n=\int \varphi d\mu$. In a similar way we endow $\M_{\le 1}(\g)$ with the \emph{vague topology}. A sequence $(\mu_n)_{n}$ converges vaguely to $\mu$ if for every $\varphi\in C_c(X)$, we have $\lim_{n\to+\infty}\int \varphi d\mu_n=\int \varphi d\mu.$

\begin{remark}In the compact case the narrow topology is the weak$^*$ topology. Clasically, the weak$^*$ topology is dual to the space of bounded and continuous functions. In the non-compact case we use the narrow topology, since the space of bounded and continuous functions is not necessarily dual to the space of probability measures. 
\end{remark}

In the non-compact case $\M(\g)$, endowed with the narrow topology, is typically not compact. For this reason we work mostly with $\M_{\le1}(\g)$, endowed with the vague topology, which makes it a compact metric space (see \cite[Corollary 13.31]{kl}). 


The relation between the narrow and vague topologies is given by the next simple and standard fact (see \cite[Theorem 13.16]{kl}).

\begin{proposition} \label{weakvag} Let $(\mu_n)_n$ be a sequence of measures in $\M(\g)$ and $\mu\in \M_{\le1}(\g)$. Then $(\mu_n)_n$ converges to $\mu$ in the narrow topology if and only if $(\mu_n)_n$ converges to $\mu$ in the vague topology and $\|\mu\|=1$.
\end{proposition}
In other words, the only obstruction to converge in the narrow topology, is the possible escape of mass.  

\begin{remark}\label{ex:seq} It is proved in \cite[Theorem 4.19]{v2} that, if $(M,g)$ is not convex cocompact, then there exist sequences in $\M(\g)$ converging vaguely to the zero measure. In particular, a manifold $(M,g)$ is convex cocompact if and only if $\M(\g)$ is compact relative to the narrow topology. 
\end{remark}

The non-elementary assumption on $\Gamma$ ensures that the set of periodic orbits is dense in the non-wandering set of the geodesic flow. It turns out that, from the point of view of periodic measures, the same property holds on $\M(\g)$ (see Corollary 2.3 and Proposition 3.2 in \cite{cs}).

\begin{lemma}\label{prop:densitypm} The set of periodic measures is dense in $\M(\g)$ with respect to the narrow topology.
\end{lemma}
 
Given $\mu\in \M(\g)$, we denote by $h(\mu)$ the measure-theoretic entropy of the time-one map $g_1$ of the geodesic flow relative to $\mu$. We refer to the map $h:\M(g)\to \R_{\ge 0}$ as the entropy map. The \emph{topological entropy} of the geodesic flow $(X,\g)$ is defined by
$$h_{top}(\g)=\sup\{h(\mu):\mu\in\M(\g)\}.$$  

The next result relates the topological entropy of the geodesic flow with the geometry of the manifold. In the compact case this is due to Manning \cite{ma} and was later extended by Sullivan \cite{su} for hyperbolic geometrically finite manifolds. The general case was proved by Otal and Peign\'e in \cite{op}.

\begin{theorem}\label{thm:op} Let $M=\widetilde{M}/\Gamma$ be a complete Riemannian manifold with pinched negative sectional curvatures. Assume $\Gamma$ is non-elementary. Then
$$h_{top}(\g)=\delta(\Gamma).$$
\end{theorem}

Another important player in the ergodic theory of the geodesic flow is the \emph{entropy at infinity}. Roughly speaking, the entropy at infinity measures the chaoticity of the system on arbitrarily small neighborhoods of infinity. For the sake of simplicity, we define it in similar terms to the topological entropy. For a geometric definition of the entropy at infinity (analogous to the critical exponent of $\Gamma$) we refer the reader to \cite{st,v2}. 

\begin{definition}\label{varprinf} The topological entropy at infinity $h_\infty(\g)$ of $(X,\g)$ is defined by
$$h_\infty(\g)=\sup_{(\mu_n)_n\to 0}\limsup_{n\to +\infty} h_{\mu_n}(g),$$
where the supremum runs over sequences $(\mu_n)_n$ converging vaguely to the zero measure.  
\end{definition}

As explained in Remark \ref{ex:seq}, if the non-wandering set of the geodesic flow is non-compact, then $h_\infty(\g)$ is well defined and lies in $[0,h_{top}(\g)]$. If $(M,g)$ is convex cocompact the entropy at infinity has no significance since $\M(\g)$ is compact and sequences of measures do not lose mass. 

We end this section recalling an important semicontinuity property of the entropy map. It is well-known in the compact case that the entropy map is upper-semicontinuous relative to the narrow topology. Since we are mostly interested in the non-compact case we need to take into consideration the escape of mass, as this is possible to occur. In this generality the next result was proved in \cite{v2}.

\begin{theorem}\label{A} Let $(M,g)$ be a pinched negatively curved manifold. Let $(\mu_n)_{n}$ be a sequence in $\M(\g)$ which converges to $\mu$ in the vague topology. Then 
$$\limsup_{n\to +\infty}h(\mu_n)\le \|\mu\|h(\overline{\mu})+(1-\|\mu\|)h_\infty(\g).$$
If $(\mu_n)_n$ converges vaguely to the zero measure, then the right hand side is understood as $h_\infty(\g)$. 
\end{theorem}

\subsection{Thermodynamic formalism} 
We start with the definition of the topological pressure of a potential $\phi:X\to \R$. This generalizes the definition of topological entropy of the geodesic flow. 

\begin{definition}[Topological pressure] Let $\phi:X\to \R$ be a continuous potential. The \emph{topological pressure} of $\phi$ is defined by 
$$P(\phi)=\sup_{\mu\in \M(\g)}\left\{h(\mu)+\int \phi d\mu:\int \phi d\mu>-\infty\right\}.$$
\end{definition}

In this paper we will be mostly interested in potentials which are bounded. This is a reasonable assumption as the topological entropy of the geodesic flow is finite. In this case, the assumption $\int \phi d\mu>-\infty$ can be omitted. 

A measure $m\in\M(\g)$ is said to be an \emph{equilibrium state} for $\phi:X\to\R$ if it realizes the supremum in the definition of topological pressure. In other words, if we have $P(\phi)=h(m)+\int \phi dm,$ and $\int \phi dm>-\infty$. 

The following class of functions play an important role in this work.

\begin{definition}\label{C_0} We say that a potential $\phi:X\to\R$ \emph{vanishes at infinity} if for every $\varepsilon>0,$ there exists a compact set $K_\varepsilon$ such that $\sup_{x\in X\setminus K_\varepsilon} |\phi(x)|<\varepsilon.$ The space of continuous functions vanishing at infinity is denoted by $C_0(X)$. 
\end{definition}

Note that $C_0(X)$ is the closure of $C_c(X)$ in $C_b(X)$ with respect to the uniform norm. The next simple lemma states that $C_0(X)$ is the space of test functions for the vague topology (see \cite[Lemma 2.20]{v2}).

\begin{lemma}\label{c0top} Let $\phi\in C_0(X)$. Then the map $\mu\mapsto \int \phi d\mu,$ is continuous in $\M_{\le 1}(\g)$ relative to the vague topology. 
\end{lemma}

A central problem in thermodynamic formalism is to understand under which assumptions it is possible to ensure the existence and uniqueness of equilibrium states. For convex cocompact manifolds this has been well understood for long time: H\"older-continuous potentials have a unique equilibrium state as consequence of the symbolic coding constructed by Bowen \cite{bo2} and Ratner \cite{ra}. In our setting the analogous result was proved in \cite{pps} (for $\phi=0$ this was proved in \cite{op}). 

\begin{theorem}\label{thm:ppsES}  Let $\phi$ be a H\"older-continuous potential. Then there exists at most one equilibrium measure for $\phi$.
\end{theorem}
 
In the non-compact setting it is possible that some potentials do not have an equilibrium state (for instance see \cite{v1}). This is true even for the potential constant equal to zero (see \cite{peigne}). A sharp criterion for the existence of equilibrium states for H\"older-continuous potential is presented in \cite{ps}. For potentials in $C_0(X)$ we have the following, easier to state, criterion for the existence of an equilibrium state. We emphasize that potentials in $C_0(X)$ are uniformly continuous, but not necessarily H\"older-continuous. 

\begin{lemma}\label{lem:exx} Let $\phi\in C_0(X)$ be a potential such that $P(\phi)>h_\infty(\g)$. Then there exists an equilibrium measure for $\phi$.
\end{lemma}

We say that $\phi\in C_0(X)$ is \emph{strongly positive recurrent} (or SPR) if $P(\phi)>h_\infty(\g)$. A more general definition of strongly positive recurrent potentials is given in \cite[Definition 2.15]{v2}. In the context of countable Markov shifts SPR potentials were introduced and studied in \cite{sa}. Because of the similarities between the thermodynamic formalism of countable Markov shifts and the geodesic flow, one expects that properties of SPR potentials known for countable Markov shifts also hold in our set up.  

We say that $(M,g)$ is SPR if the potential constant equal to zero is SPR, or equivalently, if $h_{top}(\g)>h_\infty(\g)$. SPR manifolds have been studied in \cite{st}, where the authors proved a formula for the rate of change of the topological entropy under suitable perturbations of the metric.

\section{Ergodic optimization and maximizing measures}\label{sec3}
In this section we investigate under which conditions we can guarantee  the existence of a maximizing measure for uniformly continuous potentials. Our main technical ingredient is Theorem \ref{main:a}, which we prove in this section. In the non-compact setting it is fairly easy to construct potentials which do not admit a maximizing measure (see Example \ref{ex:nomea}), and a suitable assumption is required on the potential to rule out such cases. In contrast, for a compact metric space, the existence of maximizing measures for a continuous potential is a well known and simple fact. For completeness we provide a proof of this fact in the case where $(M,g)$ is convex cocompact. 

\begin{proposition}\label{felipito} Let $(M,g)$ be convex cocompact. Then every potential $\phi\in C_b(X)$ admits a maximizing measure.
\end{proposition}
\begin{proof} Since $(M,g)$ is convex cocompact, the non-wandering set of the geodesic flow is compact. This implies that $\M(\g)$, endowed with the narrow topology, is a compact metric space. The function $A:\M(\g)\to \R$, given by $A(\mu)=\int \phi d\mu$ is continuous, and therefore it  has a maximum. 
\end{proof}

For the rest of this section we will always assume that the non-wandering set of the geodesic flow is non-compact. Recall that for a continuous potential $\phi$ we have $\beta(\phi)=\sup_{\mu\in \M(\g)}\int \phi d\mu,$ and that, in analogy to the entropy at infinity of the geodesic flow (see Definition \ref{varprinf}), we defined $\beta_\infty(\phi)$ (see Definition \ref{def:Minf}). The basic relation between $\beta(\phi)$ and $\beta_\infty(\phi)$ is given by the following result. 

\begin{lemma}\label{lem:a} Let $\phi\in C(X)$ and assume $\beta(\phi)<+\infty$. If $\psi\in C_0(X)$ is a positive function, then
$$\beta_\infty(\phi)=\lim_{t\to+\infty}\beta(\phi-t\psi).$$
\end{lemma}
\begin{proof}  Observe that $h(t):=\beta(\phi-t\psi)$ is a decreasing function, in particular $\lim_{t\to+\infty} \beta(\phi-t\psi)$ exists. Note that the limit could be equal to $-\infty$. 

We will first prove that $\beta_\infty(\phi)\le \beta(\phi-t\psi)$, for every $t\in \R$. Let $(\mu_n)_n$ be a sequence in  $\M(\g)$ which converges vaguely to zero and such that $\lim_{n\to+\infty}\int \phi d\mu_n=\beta_\infty(\phi)$. By definition of $\beta(\cdot)$ we know that 
$$\limsup_{n\to+\infty} \int (\phi-t\psi)d\mu_n\le \beta(\phi-t\psi).$$
Since the sequence $(\mu_n)_n$ converges vaguely to zero and $\psi\in C_0(X)$, we have $\lim_{n\to+\infty}\int \psi d\mu_n=0$. In particular, we get 
$$\beta_\infty(\phi)=\lim_{n\to+\infty}\int \phi d\mu_n=\lim_{n\to+\infty} \int (\phi-t\psi)d\mu_n\le \beta(\phi-t\psi).$$
This readily implies that $\beta_\infty(\phi)\le \lim_{t\to+\infty}\beta(\phi-t\psi)$. 

We will now prove that $\lim_{t\to+\infty}\beta(\phi-t\psi)\le \beta_\infty(\phi)$. Choose $\nu_n\in \M(\g)$ such that $\int (\phi-n\psi) d\nu_n\ge \beta(\phi-n\psi)-\frac{1}{n}.$ In particular 
\begin{align}\label{a1} \liminf_{n\to+\infty}\int (\phi-n\psi)d\nu_n \ge \lim_{n\to+\infty}\beta(\phi-n\psi)\ge \beta_\infty(\phi),\end{align}
where the last inequality was obtained in the paragraph above. We now separate the analysis into two cases:\\

Case 1 (when $\beta_\infty(\phi)=-\infty$): If $(\nu_n)_n$ does not converge vaguely to zero, then $\limsup_{n\to+\infty}n\int\psi d\nu_n=+\infty$, and the left hand side in \eqref{a1} converges to $-\infty$. If $(\nu_n)_n$ converges vaguely to zero, then by the assumption $\beta_\infty(\phi)=-\infty$, we have $\lim_{n\to+\infty}\int \phi d\nu_n=-\infty,$ and we conclude that the left hand side of \eqref{a1} converges to $-\infty$. In both situations we conclude that $\lim_{n\to+\infty}\beta(\phi-n\psi)=-\infty$\\

Case 2 (when $\beta_\infty(\phi)$ is finite): Observe that $(\nu_n)_n$ converges vaguely to zero, otherwise $\limsup_{n\to+\infty}n\int \psi d\nu_n=+\infty,$ which would imply that the left hand side in (\ref{a1}) goes to $-\infty$, and contradicts that $\beta_\infty(\phi)$ is finite. We conclude that 
$$\liminf_{n\to+\infty}\int \phi d\nu_n\ge  \liminf_{n\to+\infty}\int (\phi-n\psi)d\nu_n\ge \lim_{n\to+\infty}\beta(\phi-n\psi)\ge \beta_\infty(\phi).$$
Since the sequence $(\nu_n)_n$ goes vaguely to zero, the left hand side is at most $\beta_\infty(\phi)$. In particular we get that $\beta_\infty(\phi)\ge \lim_{n\to+\infty}\beta(\phi-n\psi)\ge \beta_\infty(\phi)$. 
\end{proof}

The key ingredient in the proof of Theorem \ref{main:a} is our next proposition, which together with Lemma \ref{lem:a} will imply the result. 

\begin{proposition}\label{prop:premain} Let $\phi$ be a bounded from above and uniformly continuous potential. Let $(\mu_n)_n$ be  a sequence of ergodic measures in $\M(\g)$ which converges vaguely to $\mu\in \M_{\le1}(\g)$. Then 
\begin{align}\label{ar} \limsup_{n\to+\infty}\int \phi d\mu_n\le \int\phi d\mu+(1-\|\mu\|)\beta(\phi).\end{align}
\end{proposition}

Before starting the proof of this proposition, we recall some useful dynamical properties of the geodesic flow. First, recall that the geodesic flow restricted to the non-wandering set is transitive (see for instance \cite[Proposition 4.7.3]{eb}). In particular, given a compact set $K\subseteq X$ and $\delta>0$, there exists $L=L(\delta,K)>0$ such that for every pair of points $x,y\in K$ there exists $s=s(x,y)\in [0,L]$ and $w=w(x,y)\in K$ such that $d(x,w)< \delta/4$ and $d(y,g_s(w))<\delta/4$. Second, by the local product structure and the closing lemma, there exists $N=N(\delta,K)$ such that the following holds: if $x\in K$ and $g_n(x)\in K$, where $n\geq N$, then there exists $z=z(x,n,\delta)$ such that 
$$d(g_t(x),g_t(z))<\delta/2, \ \ \text{for every }t\in[0,n],$$
$$d(g_t(w),g_{n+t}(z))<\delta/2, \ \ \text{for every }t\in[0,s], \text{ and}$$
$$g_m(z)=z,$$
where $|m-(n+s)|<\delta/2$, $w=w(g_n(x),x)$ and $s=s(g_n(x),x)$. In simpler terms we have two geodesic segments, the geodesic segment obtained  by flowing time $n$ from $x$ and the geodesic segment obtained by flowing time $s$ from $w$. We use the local product structure to shadow the two segments, and then use the closing lemma to get a periodic orbit based at $z$, as described above.

\begin{proof}[Proof of Proposition \ref{prop:premain}]  We will first deal with the case where $\int \phi d\mu>-\infty$. Note that since $\phi$ is bounded from above, we have $\beta(\phi)<+\infty$. It is enough to prove the inequality when  $|\mu|>0$, otherwise the result follows from the fact $\beta_\infty(\phi)\le \beta(\phi)$. We will moreover assume that $\beta(\phi)\ge 0$, otherwise we can replace $\phi$ by $\phi+C$, for a sufficiently large constant $C$. 

Fix a large compact set $K$ of $X$ contained in the non-wandering set of the geodesic flow such that $\mu(\partial K)=0$, and $\mu(K)>0$.

Let $Y$ be the set of points in $X$ that visit $K$ under non-positive and non-negative times. For $x\in Y$ we define $n_1(x)$ and $n_2(x)$ as the smallest non-negative real numbers such that $g_{-n_1(x)}x\in K$,  and $g_{n_2(x)}(x)\in K$. For $x\in Y$ we define $n(x):=n_1(x)+n_2(x)$. For $r\in [0,+\infty)$ define the sets 
\begin{align*}N_r=\{x\in Y: n(x)=r\}, \text{ }A_r=\bigcup_{t\in [0,r)}N_t \text{ and }B_r=\bigcup_{t\in [r,+\infty)}N_t.\end{align*}
Let $\cQ$ be the subset of $[0,+\infty)$ such that $\mu(N_m)=0$ for every $m\in \cQ$. Observe that if $x\in \partial A_m$, then $x\in g_{n_1(x)}\partial K \cup g_{-n_2(x)}\partial K\cup N_m$. In particular $\mu(\partial A_m)=0$ for every $m\in \cQ$. The set $\cQ$ is equal to $[0,+\infty)$ after removing a countable set. Observe also that the set $A_r$ is bounded for every $r\in [0,+\infty)$, so in particular 
$$I_r:=\inf_{x\in A_r}\phi(x) \ \ \text{   and   } \ \ M_r:=\sup_{x\in A_r}\phi(x),$$ are finite. 

For $x\in K$, we say that $[p,q)$ is an \emph{excursion of $x$ into} $B_r$ if $g_t(x)\in B_r$ for every $t\in (p,q)$, $g_p(x)\in K$ and $g_q(x)\in K$.   

Fix $\e>0$. Since $\phi$ is uniformly continuous, there exists $\delta_0>0$ such that if $d(x,y)<\delta_0$, then $|\phi(x)-\phi(y)|<\e$. We set $L_0=L(\delta_0,K)$ and $N_0=N(\delta_0,K)$ as in the paragraph below the statement of Proposition \ref{prop:premain}.\\ 

Let $\eta\in \M(\g)$ be an ergodic measure such that  $\phi$ is $\eta$-integrable and $\eta(K)>0$. Fix $x\in K$ a generic point of $\eta$. By the Birkhoff ergodic theorem we have
$$\lim_{t\to+\infty}\frac{1}{t}\int_0^t \phi(g_s(x))ds=\int \phi d\eta,\text{ and }  \lim_{t\to+\infty}\frac{1}{t}\int_0^t \phi(g_s(x)) 1_{A_m}(g_s(x))ds=\int_{A_{m}}\phi d\eta,$$
where $1_S$ is the characteristic function of $S\subseteq X$ and $m\ge N_0$. We decompose the interval $[0, +\infty)$ in terms of excursions of $x$ into $B_{m}$. More precisely, 
$$[0,+\infty)=[a_1,b_1)\cup [b_1,c_1)\cup[a_2,b_2)\cup [b_2,c_2)\cup\ldots, $$
 where $[b_i,c_i)$ is an excursion of $x$ into $B_{m}$ and $[a_i,b_i)$ are the complementary intervals. 
Note that $g_{b_i}(x)\in K$ and $g_{c_i}(x)\in K$, and that $c_i-b_i\ge m$. In particular we can consider the periodic point $z_i:=z(g_{b_i}(x),c_i-b_i, \delta_0)$, where $g_{m_i}(z_i)=z_i$, and $|m_i-(c_i-b_i +s_i)|<\delta_0/2$. Here $s_i=s(g_{b_i}(x),g_{c_i}(x))$. Moreover, since $\phi$ is uniformly continuous, we have that
\begin{align}\label{b1}
\int_{b_i}^{c_i}\phi(g_{t}(x))dt< \int_{0}^{c_i-b_i}\phi(g_t(z_i))dt+\e (c_i-b_i). 
\end{align}
Since 
\begin{align}\label{b2}
m_i-(c_i-b_i)< s_i+\frac{\delta_0}{2}\le L_0+\frac{\delta_0}{2},
\end{align}
we get that $g_t(z_i)\in A_{L_0+\delta_0}$, for every $t\in [c_i-b_i,m_i]$. We conclude that
\begin{align}\label{b3}
R:=\min\{0,(L_0+\delta_0/2)I_{L_0+\delta_0}\}\le \int_{c_i-b_i}^{m_i}\phi(g_t(z_i))dt. 
\end{align}
Combining (\ref{b1}), (\ref{b2}) and (\ref{b3}) we get that 
\begin{align*}
\int_{b_i}^{c_i}\phi(g_t(x))dt&<\int_0^{m_i}\phi(g_t(z_i))dt+\e(c_i-b_i)-R\\
&\le m_i \beta(\phi)+\e(c_i-b_i)-R\\
&\le (c_i-b_i)(\beta(\phi)+\e)+(L_0+\delta_0/2)\beta(\phi)-R.
\end{align*}
Therefore
\begin{align*}\frac{1}{c_n}\int_{0}^{c_n}\phi(g_t(x))dt&\le\frac{p_n}{c_n} (\beta(\phi)+\e)+\frac{n}{c_n}C+\frac{1}{c_n}\sum_{i=1}^n\int_{a_i}^{b_i}\phi(g_t(x))dt\\
&=\frac{p_n}{c_n} (\beta(\phi)+\e)+\frac{n}{c_n}C+\frac{1}{c_n}\int_0^{c_n}\phi(g_t(x))1_{A_{m}}(g_t(x))dt
\end{align*}
where $p_n:=\sum_{i=1}^n(c_i-b_i)$, $C:=(L_0+\frac{\delta_0}{2})\beta(\phi)-R$, and we used that the intervals $(a_i,b_i)$ are precisely the times $t$ where $g_t(x)$ belongs to $A_{m}$. Notice that $$\lim_{n\to+\infty}\frac{p_n}{c_n}=\lim_{n\to+\infty}\frac{1}{c_n}\int_0^{c_n}1_{B_{m}}(g_s(x))ds=\eta(B_{m}).$$ 
Since $c_n\ge \sum_{i=1}^{n}(c_i-b_i)\ge n m$, then $ \frac{1}{m}\ge \frac{n}{c_n}$. Sending $n$ to infinity we get that 
\begin{align}\label{ax}
\int \phi d\eta\le \eta(B_{m})(\beta(\phi)+\e)+\frac{1}{m}C+\int_{A_{m}}\phi d\eta.
\end{align}

We now return to the sequence $(\mu_n)_n$. Since $(\mu_n)_n$ converges vaguely to $\mu$ we know that $\mu_n(K)>0$, for large enough $n$. Without loss of generality we can assume that $\mu_n(K)>0$, and that $\phi$ is $\mu_n$-integrable for all $n\in \N$ (otherwise there is nothing to prove). 
It follows from (\ref{ax}), that 
\begin{align}\label{b4}
\int \phi d\mu_n\le \mu_n(B_{m})(\beta(\phi)+\e)+\frac{1}{m}C+\int_{A_{m}}\phi d\mu_n,
\end{align}
for every $n\in\N$.  Since $A_m$ is relatively compact and $\mu(\partial A_m)=0,$ for every $m\in \cQ,$ we conclude that 
$$\lim_{n\to+\infty}\int_{A_m}\phi d\mu_n=\int_{A_m}\phi d\mu, \text{ and }\lim_{n\to+\infty}\mu_n(B_m)=1-\mu(A_m).$$ Combining this with (\ref{b4}) we get
\begin{align*} \limsup_{n\to+\infty}\int \phi d\mu_n\le (1-\mu(A_{m}))(\beta(\phi)+\e)+\frac{1}{m}C+\int_{A_{m}}\phi d\mu,\end{align*}
for every $m\in\cQ\cap [n_0,+\infty)$. Finally take $m$ to infinity and note that $C$ is a constant independent of $m$ (it only depends on $\phi$, $\e$ and $K$). Since $\e>0$ was arbitrary this concludes the proof of \eqref{ar}.

It remains to consider the case where $\int \phi d\mu=-\infty$. For $h<0$ we define the potential $\phi_h:X\to\R,$ by $\phi_h(x)=\max\{\phi(x),h\}$. Since $\int \phi_h d\mu>-\infty$ we can use what we proved in the case above for the potential $\phi_h$, more precisely, that
$$\limsup_{n\to+\infty}\int\phi_h d\mu_n\le \int\phi_h d\mu+(1-\|\mu\|)\beta(\phi_h).$$
Combining the inequality above with 
$$\limsup_{n\to+\infty}\int \phi d\mu_n\le \limsup_{n\to+\infty}\int\phi_h d\mu_n, \text{ and }\lim_{h\to -\infty}\int \phi_h d\mu=-\infty,$$
we conclude that $\limsup_{n\to+\infty}\int \phi d\mu_n=-\infty$, as desired.  
\end{proof} 

\begin{remark}\label{nonergodiccase} For bounded potentials it follows from the narrow-density of ergodic measures that Proposition \ref{prop:premain} holds for sequences of non-ergodic measures. For general bounded from above potentials the same can be obtained adapting the proof of \cite[Proposition 6.14]{iovel} to our setting. 
\end{remark}

\begin{proof}[Proof of Theorem \ref{main:a}] Let $\psi\in C_0(X)$ be a positive function. By Proposition \ref{prop:premain} we know that $$\limsup_{n\to+\infty}\int (\phi-t\psi)d\mu_n\le \int (\phi-t\psi)d\mu+(1-\|\mu\|)\beta(\phi-t\psi),$$
for every $t\in \R$.  Since $\psi\in C_0(X)$ we know that $\lim_{n\to+\infty}\int \psi d\mu_n=\int \psi d\mu$. We conclude that 
$$\limsup_{n\to+\infty}\int \phi d\mu_n\le \int \phi d\mu+(1-\|\mu\|)\beta(\phi-t\psi),$$
for every $t\in \R$. Send $t$ to infinity and use Lemma \ref{lem:a} to conclude the proof of the theorem.
\end{proof}

\begin{remark}
In virtue of Remark \ref{nonergodiccase} it is possible to prove that Theorem \ref{main:a} is sharp. Let $\mu\in \M(\g)$ and $\lambda\in [0,1]$. We claim it is possible to construct a sequence of probability measures which converges vaguely to $\lambda \mu$ and realize the equality in Theorem \ref{main:a}. By definition there exists a sequence $(\mu_n)_n$ in  $\M(\g)$ which converges vaguely to the zero measure and such that $\lim_{n\to+\infty}\int \phi d\mu_n=\beta_\infty(\phi)$. Define $\nu_n=\lambda\mu+(1-\lambda)\mu_n$, and notice that $(\nu_n)_n$ converges vaguely to $\lambda \mu$, and 
$$\lim_{n\to+\infty}\int \phi d\nu_n=\lambda \int \phi d\mu+(1-\lambda)\beta_\infty(\phi).$$ 
\end{remark}

Our next result and its corollary follow directly from Theorem \ref{main:a} and the compactness of $\M_{\le 1}(\g)$ with respect to the vague topology.

\begin{proposition}\label{prop:a}Let $\phi$ be a bounded from above and uniformly continuous potential. Let $(\mu_n)_n$ be a sequence of ergodic measures in $\M(\g)$ such that 
$$ \lim_{n\to+\infty}\int \phi d\mu_n=\beta(\phi).$$
\begin{enumerate}
\item Suppose that $\beta_\infty(\phi)<\beta(\phi)$. Then every accumulation point of $(\mu_n)_n$, with respect to the vague topology, is a maximizing measure of $\phi$. 
\item If there is no maximizing measure of $\phi$, then $(\mu_n)_n$ converges vaguely to the zero measure. In particular, we get $\beta_\infty(\phi)=\beta(\phi)$. 
\item The accumulation points of $(\mu_n)_n$, with respect to the vague topology, is contained in
$$\S=\{\mu\in \M_{\le 1}(\g): \mu=\lambda \nu,\text{ where }\lambda\in[0,1]\text{ and }\nu\text{ maximizes }\phi\}.$$
Moreover, each measure in $\S$ can be obtained as the limit of a sequence $(\nu_n)_n$ so that $\lim_{n\to+\infty}\int \phi d\nu_n=\beta(\phi)$. 
\end{enumerate}
\end{proposition}
 
\begin{corollary}\label{cor:gapM} Let $\phi$ be a bounded from above and uniformly continuous potential such that $\beta_\infty(\phi)<\beta(\phi)$. Then there exists a maximizing measure for $\phi$. 
\end{corollary}

It follows from Proposition \ref{prop:a}$(2)$ that the only obstruction to the existence of a maximizing measure is the full escape of mass. This result should be compared with \cite[Theorem 5.2]{v2}, where an analogous result is obtained for the entropy map. 

We finish this section with the construction of a family of potentials that do not admit a maximizing measure (compare with Proposition \ref{felipito}). 

\begin{example}\label{ex:nomea} Assume that $(M,g)$ is not convex cocompact. Given a compact set $K_0\subseteq X$, define
$$\phi(x)=\frac{d(x,K_0)}{1+d(x,K_0)},$$
where $d(x,K)$ denotes the distance from $x$ to the compact set $K_0$. By construction, the potential $\phi$ is Lipschitz, positive and bounded from above by 1. We claim that $\phi$ does not admit any maximizing measure. First note that $\beta_\infty(\phi)\leq \beta(\phi)\leq 1.$
On the other hand, for any $\varepsilon>0$ there exists a compact set $K_\varepsilon\subseteq X$ such that $\phi(x)\geq 1-\varepsilon,$ for every $x\in X\setminus K_\varepsilon$. In particular, if $(\mu_n)_n$ is a sequence in $\M(\g)$ converging vaguely to zero, then
$$\lim_{n\to+\infty} \int \phi d\mu_n \geq \lim_{n\to+\infty}\int_{X\setminus K_\varepsilon} \phi d\mu_n \geq 1-\varepsilon.$$
Since $\varepsilon$ was arbitrary, we conclude that $\beta_\infty(\phi)\geq 1$. It follows that $\beta_\infty(\phi)=\beta(\phi)=1$. Since $\phi(x)<1$ for every $x\in X$, there is no measure in $\M(\g)$ such that $\int \phi d\mu=1$. 
\end{example}

\section{Ground states at zero temperature}\label{sec:gz}
We first define the class of potentials we will be interested in. 
\begin{definition}\label{def:pot} We say that $\phi$ belongs to $\F$ if $\phi$ is H\"older continuous, $\phi\in C_0(X)$ and $\beta(\phi)>0$. 
\end{definition} 
It follows by definition that $P(t\phi)\ge \beta(t\phi)=t\beta(\phi)$, for all $t\ge 0$. In particular, if $\beta(\phi)>0$, then $P(t\phi)>h_\infty(\g)$, for sufficiently large $t$. We define 
$$t_\phi:=\inf\{t\in \R:P(s\phi)>h_\infty(\g),\text{ for all }s> t\}.$$
By Theorem \ref{thm:ppsES} and Lemma \ref{lem:exx} we know $t\phi$ has a unique equilibrium state for all $t\in (t_\phi,+\infty)$. We denote by $m_{t\phi}$ the equilibrium state of $t\phi$ if this exists. For the rest of this section we will always assume $\phi\in \cF$.  Our next lemmas state that the family of equilibrium states $(m_{t\phi})_{t>t_\phi}$ and their entropies $(h(m_{t\phi}))_{t>t_\phi}$ vary continuously on $(t_\phi,+\infty)$

\begin{lemma}\label{lem:t0}  The map $t\mapsto m_{t\phi}$ is continuous in $(t_\phi,+\infty)$ relative to the narrow topology.
\end{lemma}
\begin{proof}
Let $t_0\in (t_\phi,+\infty)$ and $(t_n)_{n\in \N}$ a sequence such that $t_n\to t_0$ as $n\to+\infty$. We want to prove that $(m_{t_n\phi})_n$ converges to $m_{t_0\phi}$ in the narrow topology. By compactness of $\M_{\leq 1}(\g)$ relative to the vague topology, we can assume, maybe after passing to a subsequence, that $(m_{t_n\phi})_n$ converges vaguely to $\nu\in\M_{\leq 1}(\g)$. Hence, Theorem \ref{A} and Lemma \ref{c0top} imply
\begin{eqnarray*}
P(t_0\phi) &=& \lim_{n\to+\infty} P(t_n \phi) = \lim_{n\to+\infty} \bigg(h(m_{t_n\phi})+t_n\int \phi dm_{t_n\phi}\bigg)\\
&\leq& \|\nu\| h(\overline{\nu}) + (1-\|\nu\|)h_\infty + t_0 \int \phi d\nu\\
&=& \|\nu\| \left( h(\overline{\nu}) + t_0\int \phi d\overline{\nu}\right) + (1-\|\nu\|)h_\infty\\
&\leq& \|\nu\| P(t_0\phi) + (1-\|\nu\|)h_\infty.
\end{eqnarray*}
Since $t_0>t_\phi$ we have $P(t_0\phi)>h_\infty$, and therefore $\|\nu\|=1$. In particular, by Proposition \ref{weakvag}, the sequence $(m_{t_n\phi})_n$ converges to $\nu$ in the narrow topology. Moreover, the measure $\nu$ is an equilibrium measure for the potential $t_0\phi$, so by Theorem \ref{thm:ppsES} we conclude that $\nu=m_{t_0\phi}$. Since the argument holds for any subsequence of $(t_n)_n$ we conclude the result.
\end{proof}

\begin{proposition}\label{prop:t1} The map $t\mapsto h(m_{t\phi})$ is continuous in $(t_\phi,+\infty)$.
\end{proposition} 
\begin{proof}
Observe that $h(m_{t\phi})=P(t\phi)-t\int \phi dm_{t\phi}$, for every $t\in (t_\phi,+\infty)$. Then use the continuity of the pressure map and Lemma \ref{lem:t0}. 
\end{proof}

We now study the behavior of the sequence $(m_{t\phi})_{t>t_\phi}$ and the corresponding entropies as $t$ goes to infinity. It is convenient to define the following set
$$\cG_\phi:=\{\mu\in \M_{\le 1}(\g):\exists (t_n)_n\nearrow +\infty,\text{ such that }(m_{t_n\phi})_n\text{ converges vaguely to }\mu\}.$$
An element in $\cG_\phi$ is called a \emph{ground state at zero temperature} of $\phi$. We will first prove that every measure in $\cG_\phi$ is a maximizing measure for $\phi$. In the compact case this is a standard fact; in our context this is not automatic as in general a potential might not even have a maximizing measure. We first rule out the possible escape of mass. 

\begin{lemma}\label{lem:t1} If $\mu\in \cG_\phi$, then $\|\mu\|=1$. 
\end{lemma}
\begin{proof} Let $(t_n)_n$ be an increasing sequence of real numbers converging to $+\infty$ such that $m_{t_n\phi}$ converges vaguely to $\mu$. By Lemma \ref{c0top}, we have 
$$\int \phi d\mu = \lim_{n\to+\infty} \int \phi dm_{t_n\phi}.$$
The assumption $\beta(\phi)>0$ and the convexity of the pressure map implies that the map $t\mapsto \int \phi dm_{t\phi}$ is strictly increasing in $(s,+\infty)$ for some sufficiently large $s$. In particular, $\int \phi d\mu$ is positive and $\|\mu\|>0$. On the other hand, the variational principle implies that
$$h(\overline{\mu})+t_n\int \phi d\overline{\mu} \leq h(m_{t_n\phi}) + t_n\int \phi dm_{t_n\phi},$$
where $\overline{\mu}$ is the normalization of $\mu$. Dividing by $t_n$ above and taking limit as $n\to+\infty$ we conclude that $\int \phi d\overline{\mu} \leq \int \phi d\mu,$ and therefore $\|\mu\|=1$. 
\end{proof}

\begin{lemma}\label{lem:t2} A measure $\mu\in \cG_\phi$ is a maximizing measure for $\phi$. 
\end{lemma}
\begin{proof}
Note first that $\beta_\infty(\phi)=0$ since $\phi$ vanishes at infinity. Therefore, since we are assuming $\beta(\phi)>0$, the conclusion (1) in Proposition \ref{prop:a} implies that there exists $\nu\in\M(\g)$ such that $\int \phi d\nu=\beta(\phi).$ Let $(t_n)_n\nearrow + \infty$ be such that $(m_{t_n\phi})_n$ converges vaguely to $\mu$. Then, by  the variational principle, we get
$$h(\nu) + t_n\int \phi d\nu \leq  h(m_{t_n\phi})+t_n\int \phi dm_{t_n\phi}.$$
Dividing by $t_n$    and taking limit as $n\to+\infty$, we obtain $\beta(\phi)\leq \int \phi d\mu$. 
Since the other inequality follows by definition (see Lemma \ref{lem:t1}), we conclude that $\mu$ is a maximizing measure for the potential $\phi$.
\end{proof} 

\begin{proposition}\label{prop:entropyGS} All ground states at zero temperature have the same measure-theoretic entropy. Moreover, the function $t\mapsto h(m_{t\phi})$ converges as $t\to+\infty$ and  $$h(\mu)=\lim_{t\to+\infty}h(m_{t\phi}),$$ for all $\mu\in \cG_\phi$.
\end{proposition}
\begin{proof} Let $\mu\in\cG_\phi$ and $(t_n)_n\nearrow+\infty$ so that $(m_{t_n\phi})_n$ converges vaguely to $\mu$.  By variational principle and maximality of $\mu$ (see Lemma \ref{lem:t2}), we get
$$h(\mu)+t\int \phi d\mu \leq h(m_{t\phi})+t\int \phi dm_{t\phi} \leq h(m_{t\phi})+t\int \phi d\mu.$$
Therefore  $h(\mu)\leq h(m_{t\phi})$, for all $t\in (t_\phi,+\infty)$. We conclude that 
$$h(\mu)\le \liminf_{t\to+\infty}h(m_{t\phi}).$$ 
Combining this with the upper semi-continuity of the entropy map we obtain 
$$\limsup_{n\to+\infty} h(m_{t_n\phi})\le h(\mu)\le \liminf_{t\to+\infty}h(m_{t\phi}).$$
It follows that \begin{align}\label{vamo} h(\mu)=\liminf_{n\to+\infty}h(m_{t\phi})=:h,\end{align} for all $\mu\in \cG_\phi$. If $\limsup_{n\to+\infty}h(m_{t\phi})>h,$ then we can choose a sequence $(s_n)_n\nearrow+\infty$ so that $\lim_{n\to+\infty}h(m_{s_n\phi})>h$. Since $(m_{s_n\phi})_n$ has a subsequence converging to some $\mu_0\in \cG_\phi$, this would contradict \eqref{vamo} for $\mu_0$. 
\end{proof}

Note that Lemma \ref{lem:t1}, Lemma \ref{lem:t2} and Proposition \ref{prop:entropyGS} put together complete the proof of Theorem \ref{main:b}. 

In light of Proposition \ref{prop:entropyGS} we define 
$$h_\phi:=\lim_{t\to+\infty}h(m_{t\phi}),$$
and note that $h(\mu)=h_\phi,$ for all $\mu\in \cG_\phi$. The number $h_\phi$ is the largest entropy a maximizing measure can have, as the following simple lemma states. 

\begin{lemma} Let $\nu$ be a maximizing measure for $\phi$. Then $h(\nu)\leq h_\phi$.
\end{lemma}
\begin{proof} By variational principle and maximality of $\nu$, we know that
$$h(\nu)+t\int\phi d\nu\leq h(m_{t\phi})+t\int\phi dm_{t\phi}\leq h(m_{t\phi})+t\int\phi d\nu,$$
and therefore $h(\nu)\leq h(m_{t\phi})$. Use Proposition \ref{prop:entropyGS} to conclude the lemma.
\end{proof}

In general, for a given $\phi$, the calculation of $h_\phi$ is a hard problem. Our next lemma provides a simple situation where $h_\phi$ takes a simpler  form.  

\begin{lemma}\label{lem:t3} Let  $K\subseteq X$ be a  $\g$-invariant compact set. Suppose that  $\phi(x)=\| \phi\|$ if and only if $x\in K$. Then $h_\phi = h_{top}(\g|K).$ Morever, every measure $\mu\in \cG_\phi$ is a measure of maximal entropy of $(K,\g|_K)$.
\end{lemma}
\begin{proof}
By Lemma \ref{lem:t2} we know $\mu\in \cG_\phi$ is a maximizing measure of $\phi$. In particular, its support must be contained in $K$. The variational principle on $(K,\g|_K)$ gives us that  $h(\mu)\leq h_{top}(\g|K).$  Let $\nu$ be a measure of maximal entropy of the system $(K,\g|_K)$. Such a measure may not be unique, however it exists as a consequence of the upper-semicontinuity of the entropy map. Note that this measure is maximal by the assumption on $\phi$. Then for large enough $t$, we have
\begin{eqnarray*}
h(\nu)+t\int \phi d\nu \leq h(m_{t\phi})+t\int \phi dm_{t\phi} \leq h(m_{t\phi})+t\int \phi d\nu.
\end{eqnarray*}
So $h(\nu)\leq h(m_{t\phi})$. It follows by Lemma \ref{lem:t3} that $h_{top}(\g|K)=h(\nu) \leq h(\mu).$
\end{proof}

The lemma above states in particular that, whenever a potential reaches its maximum in a disjoint finite union of compact sets, the support of any ground state at zero temperature privileges compact sets having larger entropy. This phenomenon can be stated as follows.

\begin{corollary}\label{cor:t3} Let $\psi\in \cF$ be a potential whose maximum is reached precisely on a $\g$-invariant compact set $C\subseteq X$ and $G|_C>0$. Suppose that $C$ is the disjoint union of two $\g$-invariant compact sets $C^-$ and $C^+$ such that $h_{top}(\g|C^+)>h_{top}(\g|C^-)$. Then any  measure $\mu$ in $\cG_\psi$ is supported on $C^+$. In particular, for every $\varepsilon>0$ and every neighborhood $V^+$ of $C^+$ disjoint from $C^-$, there exists $T_C>0$ such that for every $t\geq T_C$, we have
$$m_{t\psi}(V^+)>1-\varepsilon.$$ 
\end{corollary}

It worth mentioning that the topological entropy of the geodesic flow can be approximated by the topological entropy of $\g$-invariant compact sets (see the proof of \cite[Lemma 6.7]{pps}, or \cite{op}). We conclude the following. 

\begin{corollary}\label{largeentropy} Ground states at zero temperature can have entropy arbitrarily close to $h_{top}(\g)$.
\end{corollary}

The two sections that follow are devoted to the proof of Theorem \ref{main:c}, that is, to the study of uniqueness and non-uniqueness  of ground states at zero temperature.

\subsection{Unique ground states at zero temperature} In this section we state a simple condition which ensures that  $(m_{t\phi})_{t>t_\phi}$ is convergent as $t\to+\infty$, or equivalently, that $\cG_\phi$  has a single element. We will use this result to prove that equilibrium states are dense in $\M_{\le1}(\g)$ and that the geodesic flow verifies the intermediate entropy property. 

\begin{proposition}\label{prop_convergence} Let $K\subseteq X$ be a $\g$-invariant compact set which supports a unique measure of maximal entropy $\mu_K$. If $\phi$ reaches its maximum precisely at $K$, then $m_{t\phi}$ converge to $\mu_K$ as $t\to+\infty$.
\end{proposition}
\begin{proof}
Use Lemma \ref{lem:t3} to conclude that $\mu_K$ is the unique element in $\cG_\phi$.\end{proof}


Recall that if $x\in X$ is a periodic point, we denote by $\mu_{\O(x)}$ the periodic measure supported on the periodic orbit $\O(x)$ (see \eqref{periodic}). 

\begin{proposition}\label{prop:GSPeriodic} Each periodic measure $\mu_{\mathcal{O}(x)}$ is the ground state at zero temperature of some positive potential $\phi\in \F$.
\end{proposition}
\begin{proof}
Set $K=\mathcal{O}(x)$ and  define $\phi:X\to\R$ by
$$\phi(x)=\frac{1}{1+d(x,K)}.$$
By construction, the potential $\phi$ is positive, 1-Lipschitz and maximal in $K$. Note that, for any $t>0$, the equilibrium measure $m_{t\phi}$ is well defined. Since $\g|_K$ is uniquely ergodic, we get the convergence of $(m_{t\phi})_t$ to $\mu_{\mathcal{O}(x)}$, as $t\to+\infty$.
\end{proof}

We say that a dynamical system $(X,T)$ has the \emph{intermediate entropy property} if for every $c\in (0,h_{top}(T))$ there exists an ergodic $T$-invariant probability measure with entropy equal to $c$, where $h_{top}(T)$ is the supremum of the entropy of $T$-invariant probability measures.

\begin{corollary}\label{cor:iep} The geodesic flow has the intermediate entropy property. 
\end{corollary}
\begin{proof}
Let $\O(x)$ be a periodic orbit of the geodesic flow and $\phi$ the potential defined  in Proposition \ref{prop:GSPeriodic}. Note that  $\cG_\phi=\{\mu_{\O(x)}\}$ and $t_\phi=0$. Observe that
$$\lim_{t\to 0} h(m_t) = h_{top}(\g),\text{ and }\lim_{t\to +\infty} h(m_t) = 0$$
so Proposition \ref{prop:t1} implies that any value in $(0,h_{top}(\g))$ is reached as entropy of some measure $m_{t\phi}$, which is precisely the intermediate entropy property since $m_{t\phi}$ is ergodic for any $t>0$. 
\end{proof}

We now prove that equilibrium states are dense in $\M(\g)$. This generalizes  the main result in \cite{Kb} to arbitrary pinched negatively curved manifolds.

\begin{corollary}\label{density} Equilibrium states are dense in $\M(\g)$ relative to the narrow topology. If $(M,g)$ is not convex cocompact, then equilibrium states are dense in $\M_{\le1}(\g)$ relative to the vague topology. 
\end{corollary}
\begin{proof} The set of periodic measures is dense in $\M(\g)$ relative to the narrow topology and if $(M,g)$ is not convex cocompact then this is still true in $\M_{\le1}(\g)$ relative to the vague topology (see Remark \ref{ex:seq}). By Proposition \ref{prop:GSPeriodic} we can approximate any periodic measure with equilibrium states, and therefore the same holds for an arbitrary invariant measure. 
\end{proof}

Ground states at zero temperature do not need to be ergodic. We now construct a potential $\phi\in\F$ having a unique non-ergodic ground state at zero temperature. 

\begin{lemma}\label{theo_convergenceNE} Let $K\subseteq X$ be a $\g$-invariant compact set which supports a unique measure of maximal entropy for $(K,\g|_K)$. Assume the existence of an order 2 symmetry $S:X\to X$ such that $K\cap S(K)=\emptyset$. Then, there exists a potential $\phi\in\F$ such that $m_{t\phi}$ converge as $t\to+\infty$ to a unique non-ergodic measure supported on $K\cup S(K)$.
\end{lemma}
\begin{proof}
Define $C=K\cup S(K)$ and set
$$\phi(x):=\frac{1}{1+d(x,C)}.$$
By construction, the potential $\phi$ attains its maximum exactly on $C$ and $\phi\in \F$. Moreover, we have $\phi\circ S = \phi$. Let $(t_n)_n\nearrow+\infty$ be a sequence of real numbers such that $(m_{t_n\phi})_n$ converge to $\mu\in \cG_\phi$ as $n\to+\infty$. Then
\begin{eqnarray*}
S_\ast \mu &=& S_\ast \left( \lim_{n\to+\infty} m_{t_n\phi}\right)=  \lim_{n\to+\infty} S_\ast m_{t_n\phi}\\
&=&  \lim_{n\to+\infty} m_{t_n(\phi\circ S)}=  \lim_{n\to+\infty} m_{t_n\phi}\\
&=& \mu.
\end{eqnarray*}
Hence, we have $S_\ast \mu =\mu$ and $h(\mu)=h_{top}(\g|C)=h_{top}(\g|K)$ by Lemma \ref{lem:t3}. In particular, we conclude that $\mu = \frac{1}{2} (\nu + S_\ast\nu)$,  where $\nu$ is the measure of maximal entropy for $(K,\g|_K)$. Note that $\mu$ is not ergodic, and since we found a formula for $\mu$, this measure is unique. 
\end{proof}

\begin{remark} Let $\iota:X\to X$ be the flip map, that is $\iota(x)=-x$, where $-(z,\vec{u})=(z,-\vec{u})$ for any $z\in M$ and $\vec{u}\in T^1_pM$. If $d_S$ denotes the Sasaki metric on $X$, then $\iota$ is an order 2 symmetry of $X$ relative to $d_S$. In particular, any negatively curved manifold admits a potential with unique non-ergodic ground state at zero temperature.
\end{remark}

\subsection{Non-unique ground states at zero temperature}\label{sec42} In this section we construct a potential $\phi\in\F$ for which $\cG_\phi$ has more than one element, or equivalently, so that the sequence $(m_{t\phi})_{t>t_\phi}$ is divergent as $t\to+\infty$. We saw in Corollary \ref{cor:t3} that ground states at zero temperature have preference to be supported on sets with large entropy (among components with largest value of the potential); this fact will play a fundamental role in the proof of the following result.


\begin{theorem}\label{theo_divergence} There exists a Lipschitz potential $\phi\in\F$ having at least two ground states at zero temperature.
\end{theorem}

The construction of such a potential follows closely the proof of Corollary 1.1 in \cite{crl}, and it is based in the following steps. First, we consider two disjoint $\g$-invariant compact sets $K^-$ and $K^+$ having the same topological entropy. Both compact sets verify
$$K^{\pm}=\bigcap_{n\geq 1} K_n^{\pm},$$
where $(K_n^{\pm})_n$ is a decreasing sequence of $\g$-invariant compact sets. The compact sets of each sequence are defined in such a way that, if we define 
\begin{equation*}\label{eq:sets}
Y_{n+1}^+=K_n^+ \cup K_{n+1}^- \ \ \text{and} \ \ Y_{n+1}^-=K_{n+1}^+ \cup K_n^-,
\end{equation*}
then
$$h_{top}(\g|Y_{n+1}^+)=h_{top}(\g|K_n^+) \ \ \text{and} \ \ h_{top}(\g|Y_{n+1}^-)=h_{top}(\g|K_n^-).$$ 
We stress the fact that the topology of the symbolic space, which is the setting in \cite{crl}, makes the construction of the compact sets $K_n^{\pm}$ easier to obtain. In our case, we have to deal with the geometry of the manifold, so this step is done in a completely different way. Second, if we consider the potentials 
\begin{equation}\label{eq:potpm}
\phi_n^{\pm}(x) = \frac{1}{1+d(x,Y_n^{\pm})},
\end{equation}
then the corresponding ground states at zero temperature will be supported on $K_{n-1}^{\pm},$ for $n\geq 2$. Third, for a sequence $(\varepsilon_k)_k$ of real numbers decreasing to 0, there exist two sequences of real numbers $(\delta_k)_k$ decreasing to 0, and $(t_k)_k$ increasing to $+\infty$, such that if
\begin{equation}\label{eq:pot}
\phi_n(x)=\sum_{k=1}^n \delta_k \phi_k^{s(k)}(x),
\end{equation}
then $\phi_n$ converges (in the Lipschitz norm) to the desired potential $\phi$. Here $s(k)$ is the symbol $+$ or $-$ depending respectively on $k$ being odd or even. We also have
\begin{equation}\label{eq:div1}
\|\phi-\phi_n\|_{Lip} < \delta_n
\end{equation}
and
\begin{equation}\label{eq:div2}
m_{t_n\phi}(U^{s(n)})>1-\varepsilon_n,
\end{equation}
where $U^-$ and $U^+$ are two given disjoint neighborhoods of $K^-$ and $K^+$.  Thus for $n$ odd the sequence $(m_{t_n\phi})_n$ has limit measures supported on $K^+$, and for $n$ even $(m_{t_n\phi})_n$ has limit measures supported on $K^-$. We remark that the role of the numbers $\delta_k$ in the definition of $\phi_n$ is to exchange the locus of maximality of the potential, and therefore to alternate the support of the measures $m_{t\phi}$ for large $t>0$.\\

Before beginning the proof of Theorem \ref{theo_divergence} we will prove two preliminary results. The first one is related to the choice of the numbers $\delta_n$ and $t_n$ mentioned above, and the second  to the existence of the compact sets $K^\pm$ and $K_n^\pm$. 

For a potential $\phi:X\to\R$, we have $\|\phi\|_\infty= \sup_{x\in X}|\phi(x)|.$  Define
$$|\phi|_{Lip} := \sup\left\{ \frac{|\phi(x)-\phi(x')|}{d(x,x')} \ : \ x,x'\in X, \ x\neq x' \right\},$$
and 
$$\|\phi\|_{Lip}:= \|\phi\|_\infty + |\phi|_{Lip}.$$
A potential $\phi:X\to\R$ is Lipschitz continuous if $\|\phi\|_{Lip}<+\infty$. Finally note that the space of Lipschitz continuous potentials with the norm $\|\cdot\|_{Lip}$ is a Banach space.

\begin{lemma}\label{lem:div} Let $\psi\in\F$ be a Lipschitz potential and let $\tau>t_\psi$ be given. Then for any $\varepsilon>0$ and $f\in C_0(X)$ there exists $\sigma>0$ such that, whenever
$$\|\psi-\varphi\|_{Lip}<\sigma, \ \varphi\in\F,$$
we have
$$\left| \int f dm_{\tau\psi} - \int f dm_{\tau \varphi} \right|<\varepsilon.$$
\end{lemma}
\begin{proof}
We argue by contradiction and assume that there exists $\varepsilon_0>0$ and $f\in C_0(X)$ such that there is a sequence $(\varphi_n)_n$ of Lipschitz potentials in $\F$ which converge to $\psi$ and $\left| \int f dm_{\tau \psi} - \int f dm_{\tau \varphi_n} \right| \geq \varepsilon_0.$ In particular $(m_{\tau \varphi_n})_n$ does not converge vaguely to $m_{\tau\psi}$ as $n\to+\infty$. We claim that this is not possible. Indeed, since the topological pressure is continuous with respect to the uniform norm, for $\nu\in\M_{\le1}(\g)$ an accumulation measure of  $(m_{\tau \varphi_n})_n$  we have
\begin{eqnarray*}
P(\tau \psi) &=& \lim_{n\to+\infty} P(\tau \varphi_n)\\
&\leq& \nu(X)h(\overline{\nu})+(1-\nu(X))h_\infty + \tau\int \psi d\nu\\
&\leq& \nu(X)P(\tau \psi)+ (1-\nu(X))h_\infty(\g)\\
&\leq& P(\tau\psi).
\end{eqnarray*}
In the last inequality we used that since $\tau>t_\psi$, then $P(\tau \psi)>h_\infty(\g)$. Note that the inequalities above imply $\nu(X)=1$ and that $\nu=m_{\tau \psi}$, which is a contradiction.
\end{proof}

We will now construct the compact sets described below the statement of Theorem \ref{theo_divergence}. A group $G\leqslant\text{Iso}(\widetilde{M})$ is a \emph{Schottky group} generated by two hyperbolic isometries $g_1,g_2\in \text{Iso}(\widetilde{M})$ if for each $i\in\{1,2\}$ there exists a compact neighborhood $C_i^+$ of the attracting fixed point of $g_i$, and a compact neighborhood $C_i^-$ of the repelling fixed point of $g_i$, such that 
$$g_i(\partial_\infty\widetilde{M}\setminus C_i^-)\subseteq C_i^+,$$
$$g_i^{-1}(\partial_\infty\widetilde{M}\setminus C_i^+)\subseteq C_i^-,$$
and the sets $C_1^+,C_1^-,C_2^+,C_2^-$ are pairwise disjoint. For such $G\leqslant\text{Iso}(\widetilde{M})$ we say that $\widetilde{M}/G$ is a \emph{Schottky manifold}. It is a well known fact that a Schottky manifold is convex cocompact. 

\begin{proposition}\label{prop:sets} There exists $\g$-invariant compact sets $K_n^\pm$, such that 
$$K^+=\bigcap_{n\geq 1} K_n^+ \ \ \text{and} \ \ K^-=\bigcap_{n\geq 1} K_n^-$$
have the same topological entropy, and for every $n\geq 1$
 \begin{enumerate}
 \item[(i)] $K_n^-\cap K_n^+=\emptyset$,
 \item[(ii)] $K_{n+1}^\pm \subseteq K_n^{\pm}$,
 \item[(iii)] $h_{top}(\g|K_{n+1}^-)< h_{top}(\g|K_n^+)$ and $h_{top}(\g|K_{n+1}^+)< h_{top}(\g|K_n^-)$.
 \end{enumerate} 
\end{proposition}
\begin{proof}
Recall that $X=T^1\widetilde{M}/\Gamma$, where $\Gamma\leqslant\text{Iso}(\widetilde{M})$. The non-elementary assumption on $\Gamma$ implies the existence of hyperbolic isometries $h_1^-, h_2^-, h_1^+$ and $h_2^+$ such that
$$\Gamma^+=\langle h^+_1,h^+_2 \rangle, \ \ \text{and} \ \ \Gamma^-=\langle h^-_1,h^-_2 \rangle$$
are Schottky groups (see \cite{dp} for further details). Let us denote by $p^\pm: T^1\widetilde{M}/\Gamma^\pm \to X$ the natural projections on the corresponding quotient manifolds. Up to large enough powers of the isometries, we can also assume that
$$p^-(\Omega_{\Gamma^-})\cap p^+(\Omega_{\Gamma^+}) = \emptyset,$$
where $\Omega_{\Gamma^{\pm}}$ is the non-wandering set of the geodesic flow on $T^1\widetilde{M}/\Gamma^{\pm}$. More generally, define
$$\Gamma^\pm_n = \langle h^\pm_1,(h^\pm_2)^{2^n} \rangle$$
and let
$$p^\pm_n: \widetilde{X}/\Gamma_n^\pm \to X$$
be the natural projection onto $X$. Finally, set
$$K^\pm_n = p^\pm_n (\Omega_{\Gamma^\pm_n}).$$ 
Note that 
$$K^+:=\bigcap_{n\geq 1} K_n^+ \ \ \text{and} \ \ K^-:=\bigcap_{n\geq 1} K_n^-$$
are the periodic orbits induced by the hyperbolic isometries $h_1^+$ and $h_1^-$ respectively. In particular, the topological entropies of $K^-$ and $K^+$ are both equal to 0. Properties $(i)$ and $(ii)$ follow immediately from the definition of the compact sets. 

To finish the proof of Proposition \ref{prop:sets} we will need the following result.

\begin{lemma} We have
\begin{equation}\label{eq:prop:1}
h_{top}(\g|K_n^\pm)=\delta(\Gamma_n^\pm)
\end{equation}
and
\begin{equation}\label{eq:prop:2}
\delta(\Gamma_n^\pm)\searrow  0 \text{ as } n\to+\infty.
\end{equation}
\end{lemma}
\begin{proof}
Note first that there is a finite-to-one semi-conjugacy between the geodesic flow $\g:T^1\widetilde{M}/\Gamma_n^\pm\to T^1\widetilde{M}/\Gamma_n^\pm$ restricted to its non-wandering set and $\g|_{K_n^\pm}$. By \cite[Theorem 17]{bo1} we know that the topological entropy of both dynamical systems coincide, and Theorem \ref{thm:op} imply that the exact value is $\delta(\Gamma_n^\pm)$; this proves \eqref{eq:prop:1}. 

To prove \eqref{eq:prop:2} we will use Proposition \ref{prop:CE}. Note first that $\delta(\Gamma_n^\pm)>0,$ because $\Gamma_n^\pm$ is non-elementary. Moreover, since $\Gamma_n^\pm$ is convex cocompact we know it is divergent and verifies $L(\Gamma_{n+1}^\pm)\neq L(\Gamma_n^\pm)$ for every $n\in \N$, so $\delta(\Gamma_{n+1}^\pm)<\delta(\Gamma_n^\pm)$. The fact that $\delta(\Gamma_n^\pm)$ converges to $0$ as $n\to+\infty$ follows from a calculation identical to the one used in the proof of \cite[Proposition 5.3]{irv}, where we replace the parabolic isometry $p$ by a hyperbolic element $h$, and $\delta(\langle p\rangle)$ by $\delta(\langle h\rangle)$, respectively.
\end{proof}

By the lemma above, up to a re-enumeration of the compact sets we constructed, we can alternate their topological entropies as desired. This finishes the proof of Proposition \ref{prop:sets}.  
\end{proof}

\begin{proof}[Proof of Theorem \ref{theo_divergence}] Let  $K^{\pm}$ and $K_n^{\pm}$ be the $\g$-invariant compact sets constructed in Proposition \ref{prop:sets}. We also consider $U^-$ and $U^+$ disjoint open neighborhoods of $K_1^-$ and $K_1^+$ respectively. Define $Y^{\pm}_n$ by
$$Y_{n+1}^+=K_n^+ \cup K_{n+1}^- \ \ \text{and} \ \ Y_{n+1}^-=K_{n+1}^+ \cup K_n^-.$$
Note that
$$\bigcap_{n\geq 1} Y_n^+ = \bigcap_{n\geq 1} Y_n^- = K:= K^+\cup K^-.$$
Consider the family of potentials $\varphi_n^{\pm}$ defined by equation \eqref{eq:potpm} and take any sequence of real numbers $(\varepsilon_k)_k$ converging to 0. We will construct the potentials $\varphi_n$ as in \eqref{eq:pot} inductively. Set $\delta_1=1$ and assume that $\delta_1,\ldots,\delta_k$ and $t_1,\ldots,t_{k-1}$ are given. Use Corollary \ref{cor:t3} for $\psi=\varphi_k$, $C^+=K^{s(k)}_n$, $C^-=Y^{s(k)}_{k+1}\setminus K^{s(k)}_k$, $\varepsilon=\varepsilon_k/3$ and $V^+=U^{s(k)}$, to obtain $t_k = \max\{T_C,t_{k-1}+k\}$. 
For every $k\geq 1$, let $f_k^\pm\in C_c(X)$ be such that
\begin{equation}\label{eq:phi}
f_k^\pm(x)=\left\{
    \begin{array}{rl}
      1 &, \text{ if } x\in K_k^\pm \\
      0 &, \text{ if } x\in X\setminus U^\pm,
    \end{array}
  \right.
\end{equation}
and
\begin{equation}\label{eq:phi2}
\int f_k^\pm dm_{t_k\varphi_k} \geq m_{t_k\varphi_k}(U^{\pm})-\frac{\varepsilon_k}{3}.
\end{equation}
Using Lemma \ref{lem:div} for $\psi=\varphi_k$, $\tau=t_k$, $\varepsilon=\varepsilon_k/3$ and $f=f_k^{s(k)}$, we define $\delta_{k+1}=\min\{\delta_k/2,\sigma\}$. 
We will now show that the formula
$$\varphi(x)=\sum_{k=1}^{\infty} \delta_k \varphi_k^{s(k)}(x)$$
defines a potential in $\F$ which attains its maximum exactly on $K$ and that verifies \eqref{eq:div1} and \eqref{eq:div2}. Note first that, by definition, we have $|\varphi_k^{\pm}|\leq 1$ for every $k\geq 1$. Since $\delta_k\leq 1/2^{k-1}$, the function $\varphi$ is pointwise well defined and $|\varphi(x)|\leq 2$ for every $x\in X$. A similar argument together with the fact that $\varphi_k^\pm$ is 1-Lipschitz for every $k\geq 1$, shows that $\varphi$ is a 2-Lipschitz potential. Note also that $Y_{n+1}^\pm\subseteq Y_n^\pm$, so $|\varphi_k^\pm(x)|\leq |\varphi_1^\pm|$ for every $k\geq 1$. In particular, the potential $\varphi$ vanishes at infinity. Also note  that $M(\varphi)>0$,  hence $\varphi\in \F$.

To prove \eqref{eq:div1}, observe that
\begin{eqnarray*}
\|\varphi-\varphi_n\|_\infty &=& \sup_{x\in X} |\varphi(x)-\varphi_n(x)| \leq \sup_{x\in X} \sum_{k=n+1}^\infty \delta_k|\varphi_k^{s(k)}(x)|\\
&\leq& \sup_{x\in X} \delta_{n+1} \sum_{k=n+1}^\infty \frac{\delta_k}{\delta_{n+1}}|\varphi_k^{s(k)}(x)|\\
&\leq& \sup_{x\in X} \delta_{n+1} \sum_{k=n+1}^\infty \frac{1}{2^{k-n}}\leq \delta_{n+1}.
\end{eqnarray*}
A similar argument shows that in fact we have the same inequality for the Lipschitz norm. Now we just need to prove that \eqref{eq:div2} holds. Using \eqref{eq:phi} and the conclusion of Lemma \ref{lem:div} we obtain
\begin{equation}\label{eq:f1}
m_{t_n\varphi}(U^{s(n)}) = \int \mathbbm{1}_{U^{s(n)}} dm_{t_n\varphi} \geq \int \varphi_n^{s(n)} dm_{t_n\varphi} \geq  \int \varphi_n^{s(n)} dm_{t_n\varphi_n} - \frac{\varepsilon_n}{3}. 
\end{equation}
Inequalities \eqref{eq:phi2} and \eqref{eq:f1} give us
\begin{equation}\label{eq:f2}
m_{t_n\varphi}(U^{s(n)}) \geq  m_{t_n\varphi_n}(U^{s(n)}) - \frac{2\varepsilon_n}{3}. 
\end{equation} 
Finally, the conclusion of Corollary \ref{cor:t3} states that
$$m_{t_n\varphi_n}(U^{s(n)}) \geq 1-\frac{\varepsilon_n}{3},$$
which together with inequality \eqref{eq:f2} imply \eqref{eq:div2}.
\end{proof}

\end{document}